\documentclass[10pt]{article}
\usepackage{geometry}
\usepackage[utf8]{inputenc}
\usepackage{amssymb}
\usepackage{amsmath}
\usepackage{amsthm}
\usepackage{graphicx}
\usepackage{fancyhdr}
\usepackage{hyperref}
\usepackage{ytableau}
\usepackage{bbding}
\usepackage{xcolor}
\usepackage{authblk} 

\usepackage[backend=biber, sorting=none]{biblatex}
\usepackage[english]{babel}
\usepackage{tikz}
\usepackage{tikz-cd}
\usetikzlibrary{cd}

\theoremstyle{plain}
\newtheorem{thm}{Theorem}[section]
 
\newtheorem{cor}[thm]{Corollary}
\newtheorem{corollary}[thm]{Corollary}
\newtheorem{prop}[thm]{Proposition}

\theoremstyle{definition}

\newtheorem{definition}[thm]{Definition}
\newtheorem{lem}[thm]{Lemma}

\newtheorem{example}[thm]{Example}

\newtheorem{remark}[thm]{Remark}

\renewcommand{\phi}{\varphi}

\newcommand{\frameop}{\Phi\Phi^\ast}
\newcommand{\gram}{\Phi^\ast\Phi}

\newcommand{\bC}{\mathbb{C}}
\newcommand{\bF}{\mathbb{F}}
\newcommand{\bR}{\mathbb{R}}
\newcommand{\bN}{\mathbb{N}}
\newcommand{\cH}{\mathcal{H}}
\newcommand{\rar}{\rightarrow}
\newcommand{\ip}[2]{\left<#1,\,#2\right>}
\newcommand{\norm}[1]{\left\Vert#1\right\Vert}

\newcommand{\absip}[2]{\left\vert\left<#1,\,#2\right>\right\vert}

\newcommand{\rk}{\operatorname{rank}}
\newcommand{\bc}{\operatorname{Boxcomp}}
\newcommand{\Rar}{\Rightarrow}

\newcommand{\beq}{\begin{equation}}
\newcommand{\eeq}{\end{equation}}

\title{Connections Between Frames with Rational Eigensteps and Semistandard Young Tableaux}

\author[1]{Emily J.\ King}
\author[1]{Kylie Schnoor\thanks{Corresponding author: \texttt{kylie.schnoor@colostate.edu}}}
\affil[1]{Department of Mathematics, Colorado State University, Fort Collins, CO}

\date{\today}

\begin{document}
\maketitle
\begin{abstract}
In this paper, we explore a correspondence between frames with rational eigensteps and semistandard Young tableaux (SSYT), via the relation assigning a Gelfand-Tsetlin pattern to a frame via the frame's eigensteps. We will identify how certain key structures in SSYTs correlate with particular frame properties. For example, the weight of an SSYT yields the sequence of norms of any compatible frame. Additionally, this correspondence leads to a novel way to construct the eigensteps of a frame coming solely from tableaux. This is an alternative to the Top Kill algorithm which may be viewed as a combinatorial reinterpretation of the algorithm. We further employ other combinatorial techniques such as the boxcomp method to generate a ``complement" SSYT. On the frame side, this corresponds to a tight frame's Naimark complement as well as to a generalization of the Naimark complement for non-tight frames. Further research points to an analysis of equiangular tight frames and their corresponding tableaux, as well as using more combinatorial operations to further analyze frames. \textbf{Keywords:} frame theory, eigensteps, semistandard Young tableaux, Gelfand-Tsetlin pattern, Naimark complement
\end{abstract}

\section{Introduction}

Frames can be seen as generalizations of bases that have many applications in wireless communication, compressed sensing, machine learning, and beyond \cite{CKentf, GVKapplication, strohmer2003grassmannianframesapplicationscoding,papyan2020prevalence,FoRa13}. On the other hand, semistandard Young tableaux are combinatorial objects utilized often in symmetric function theory, algebraic geometry, and representation theory \cite{FultonTableaux, Stanleycombinatorics, GRgeneralized, CrystalBases, RepTheory}. The goal of this work is to present a connection between the two objects and further leverage combinatorial theory to  understand frames.

Frames may be associated to Gelfand-Tsetlin (GT) patterns via their eigensteps~\cite{cfmconstructingfiniteframes, cfpconstructingselfadjoint,HPpolytope}, while integer GT patterns are well-known in combinatorics to be equivalent to semistandard Young tableaux (SSYT) (see, e.g., \cite{Stanleycombinatorics}).  In this paper, we introduce the connection between frames and SSYTs as filtered through GT patterns and explore how combinatorial manipulations of SSYT can be interpreted on the frame side, capitalizing on finite frame theory while also drawing on combinatorial influences from \cite{FultonTableaux, GTpatterns}. Tableaux theory has been leveraged in different ways in prior work in frame theory: the rich connections between tableaux and irreducible representations of $S_n$ have been leveraged in order to construct equi-isoclinic tight fusion frames (EITFFs) as well as to provide a proof for the Grigoriev-Laurent lower bound \cite{fickus2022note, kunisky2024spectrum,}. Additionally, standard Young tableaux and the Hook length formula have been used to further analyze the symmetry of EITFFs and equi-chordal tight fusion frames (ECTFFs) \cite{fickus2024equi}. We hope this work encourages more exploration of the connections between these objects.

In Section~\ref{frametheorybackground} we discuss the relevant frame theory background. In Section~\ref{combinatorialbackground} we discuss the relevant combinatorics and introduce Gelfand-Tsetlin patterns and Young tableaux.  We further discuss eigensteps in Section~\ref{eigensteps}. In Section~\ref{semistandardframes} we establish new results that relate semistandard tableaux to the eigensteps of a frame. Specifically, we analyze how the previous results can be extended to the Naimark complement of a tight frame in Section~\ref{naimarkcomplements} and further generalize this correspondence to non-tight frames in Section~\ref{generalizednaimark}.

\subsection{Frame Theory Background}\label{frametheorybackground}
We begin with an introduction to finite frame theory. See, e.g., \cite{Waldron18}, for a general resource about finite frames.  Formally, we have:

\begin{definition}\label{framedef}
Let $\cH$ be a finite-dimensional Hilbert space. A sequence $\Phi =  \left(\phi_i\right)_{i=1}^n$ in $\cH$ is called a \emph{frame} over $\cH$ if there exist $0 < A \leq B$ such that the following holds for all $x \in \cH$:
\[
    A\norm{x}^2 \leq \sum_{i=1}^n \absip{x}{\phi_i}^2 \leq B\norm{x}^2.
\]

\end{definition}

We call $A$ and $B$ \emph{frame bounds}, with $A$ being a \emph{lower frame bound} and $B$ an \emph{upper bound}. If $A$ is the largest possible lower bound and $B$ the smallest possible upper bound, we call them \emph{optimal}. If the optimal frame bounds are equal to each other, we call the frame \emph{tight (with frame bound $A$)}. Over finite-dimensional Hilbert spaces, any spanning set is a frame; so, the optimal bounds provide key information. 

The $\phi_i$ are the \emph{frame vectors}. There are several operators from frame theory associated to sequences of vectors.  After fixing a basis, we may represent the frame vectors as elements of $\bF^d$ for $\bF = \bR$ or $\bC$.  By slight abuse of notation, we use $\Phi$ to denote both the $d \times n$ \emph{synthesis matrix}
\[
\begin{bmatrix}\varphi_1 & \varphi_2 & \hdots & \varphi_n \end{bmatrix}
\]
which maps $\bF^n \rar \bF^d$ and the frame itself.  Thus, we will also call such a frame a \emph{$(d,n)$-frame}.
 
We call $\Phi\Phi^\ast$ the \emph{frame operator} and $\Phi^\ast\Phi$ the \emph{Gram matrix}. As a function, the frame operator maps 
\[
x \mapsto \sum_{i=1}^n \ip{x}{\varphi_i}\varphi_i = (\varphi_i \varphi_i^\ast)x.
\]
On the other hand, the $(i,j)$ entry of the Gram matrix is $\ip{\varphi_j}{\varphi_i}$.

A frame is tight if and only if $\Phi\Phi^\ast = AI$ where $A$ is the frame bound and $I$ is the $d \times d$ identity matrix. 

Note that if $\Phi$ is tight with frame bound $A$, then the Gram matrix $G = \Phi^\ast \Phi$ is a self-adjoint operator satisfying
\begin{align*}
G&= (\Phi^\ast \Phi)^2= (\Phi^\ast \Phi)(\Phi^\ast \Phi) \\
&= \Phi^\ast (\Phi\Phi^\ast) \Phi =\Phi^\ast A I \Phi\\
&= A \Phi^\ast \Phi=AG.
\end{align*}
That is, $G$ is a scaled orthogonal projection. Conversely, any scaled orthogonal projection is the Gram matrix of a tight frame by spectral theory. So, a scaled orthogonal projection with diagonal entries all equal to $w>0$ must be a Gram matrix of a tight frame with all vectors having $w$ as the square of their norm.  Such a frame is called an \emph{equal norm tight frame}.

It is often very useful to discuss and analyze the eigenvalues of the frame operator and the Gram matrix. For example, the frame operator and Gram matrix have the same non-zero eigenvalues since they are the products of two matrices in both orders.  A common trick to analyze spectra in frame theory is to use the cyclic invariance of the trace, as shown in the standard proof of the following lemma.

\begin{lem}\label{lem:sumnorm}
Let $\Phi=(\varphi_i)_{i=1}^n$ be a frame. If $\lambda_1 \dots \lambda_d$ are the eigenvalues of $\Phi\Phi^\ast$, then $$ \sum^d_{j=1} \lambda_j = \sum^n_{i=1} ||\phi_i||^2.$$
So, if $\Phi$ is an equal-norm tight frame with frame bound $A$ and norm squared $w$, then $A=nw/d$.
\end{lem}

\begin{proof}
Since $\Phi\Phi^\ast$ and $\Phi^\ast\Phi$ have the same nonzero eigenvalues, they will have the same trace. Thus we have: 
$$\sum^d_{j=1} \lambda_j = \text{tr}(\Phi\Phi^\ast) = \text{tr}(\Phi^\ast\Phi) = \sum^n_{i=1} ||\phi_i||^2 $$
\end{proof}

\subsection{Combinatorial Background}\label{combinatorialbackground}
In this section we discuss the necessary combinatorial background. We begin by introducing Gelfand-Tsetlin patterns and then move on to tableaux theory. Gelfand-Tsetlin patterns and semistandard Young tableaux have a close relationship with each other; both objects have strong ties to representation theory \cite{CrystalBases, RepTheory, HKGTcrystals}. Below is the standard definition of a Gelfand-Tsetlin pattern found in \cite{Stanleycombinatorics, GTpatterns, HKGTcrystals}. \\

\begin{definition}\label{GTdef}
A \emph{Gelfand-Tsetlin (GT) pattern} is a triangular or parallelogram arrangement of non-negative numbers such that each row is weakly decreasing from left to right, weakly decreasing from the top-left to bottom-right diagonally, and weakly decreasing from the bottom-left to top-right diagonally. In other words for $((\lambda_{i,j})_{j=1})_{i=1}^n$, we have that
\begin{enumerate}
\item $\lambda_{i,j} \geq \lambda_{i,j+1}$
\item $ \lambda_{i+1,j} \geq \lambda_{i,j} \geq \lambda_{i+1,j+1}$ for $1 \leq j \leq i \leq n-1$
\end{enumerate}
whenever $i,j$ are defined.

When all $\lambda_{i,j}$ are integers, it is an \emph{integer GT pattern}.  Unless otherwise noted, GT patterns in this paper are integer GT patterns. The \emph{weight vector} of a GT pattern is $w = (w_1, w_2, \dots, w_n)$ where $w_i$ is defined as the difference in row sums between row $i$ and row $i-1$, where row $0$ is defined to be all zeros.
\end{definition}

Note that the upper bound on the range of $j$ is deliberately not defined. This is due to the fact that a GT pattern can be either a triangular shape (in which the number of entries in row $i$ will be $i$) or a parallelogram shape (in which the size of each sequence will be constant).

\begin{example}
An example of a triangle GT pattern is
\[
\begin{matrix}
4 & & 3 & & 3 & & 1 \\
& 4 & & 3 & & 2 & \\ 
& & 4 & & 2 & & \\ 
& & & 4 & & & 
\end{matrix}
\]
with weight vector $(4, 2, 3, 2)$
An example of a parallelogram GT pattern is: 
\[
\begin{matrix}
4 & & 3 & & 2 & & 1 & & \\
&4 & & 2 & & 2 & & 1 & &  \\
&&2 & & 2 & & 1 & & 1 & \\
&&&2 & & 2 & & 1 & & 1 
\end{matrix}
\]
with weight vector $(6,0,5,4)$
\end{example}

Additionally, note that a variety of GT patterns are equivalent simply by 0-padding the top row. For example, it is easy to increase the number of rows on top of a pattern by duplicating the nonzero entries in the top row (and then padding by an extra zero if it is a triangle-shaped pattern). The following GT patterns are equivalent in that sense:
\begin{align*}
\begin{matrix}
4 & & 3 & & 3 & & 1 \\
& 4 & & 3 & & 2 & \\ 
& & 4 & & 2 & & \\ 
& & & 4 & & & 
\end{matrix}
&& 
\begin{matrix}
4& & 3 & & 3 & & 1 & & 0\\
&4 & & 3 & & 3 & & 1 \\
&& 4 & & 3 & & 2 & \\ 
&& & 4 & & 2 & & \\ 
&& & & 4 & & & 
\end{matrix}
\end{align*}

In the correspondence between GT patterns and frames we will see Proposition~\ref{prop:clear2GT}, the modification used to generate the right-hand GT pattern above corresponds to appending a zero vector to a frame. Thus, in most cases, we will be dealing with the ``smallest possible'' GT pattern with the given structure.
Continuing on, another combinatorial object we will extensively use is tableaux. We begin by defining one key component:

\begin{definition}
A \emph{Young diagram} is a finite collection of boxes in left-justified rows with row lengths in non-increasing order. Listing these row lengths in non-increasing order corresponds to a integer partition often denoted $\lambda = (\lambda_1, \lambda_2, \dots, \lambda_n)$.
\end{definition}

\begin{remark}
We will be using French notation in order to denote Young diagrams and later semistandard Young tableaux. This means that the Young diagram will be bottom-left justified, instead of top-left justified (English notation) as is seen \cite{FultonTableaux}. In other words, gravity exists.
\end{remark}

\begin{definition}
A \emph{semistandard Young tableau (SSYT)} is a filling of a Young diagram with non-negative integers such that the entries of the Young diagram are weakly increasing row-wise and strictly increasing column-wise.

Given a \emph{proper filling} (a filling that satisfies the above conditions) of a SSYT using numbers $[n]=\{1, 2, \hdots, n\}$, the \emph{weight vector} is the vector $w = (w_1, w_2, \dots w_n)$ such that $w_i$ is the number of $i$'s in the tableau.

\end{definition}

\begin{example}\label{ex:SSYT}
An example of an SSYT of shape $(4,3,3,1)$ with weight vector $(4,2,3,2)$ is
\[
\begin{ytableau}
4 \\ 
3 & 3 & 4 \\
2 & 2 & 3 \\ 
1 & 1 & 1 & 1 \\ 
\end{ytableau}.
\]
\end{example}

\begin{definition}
Consider two partitions $\lambda = (\lambda_1, \lambda_2, \dots, \lambda_n)$ and $\mu = (\mu_1, \mu_2, \dots, \mu_m)$ such that $m \leq n$ and $\mu_i \leq \lambda_i$ for all $i$. A \emph{skew Young diagram} has shape $\lambda / \mu$ and results from removing the partition $\mu$ from $\lambda$. A filling of a skew tableau that is still weakly increasing in rows and strictly increasing in columns is called a \emph{skew SSYT}. 
\end{definition}

Note that a tableau that is not a skew shape is often called a \emph{straight shape tableau}. This difference will be used extensively throughout this paper. It will explicitly be stated when a tableau is of skew shape. If no distinction is made, it is safe to assume the tableau is straight shape. 

\begin{example}
Take the SSYT from Example~\ref{ex:SSYT} and skew it by the partition $(3,2,1)$, then the resulting skew tableau is:  
\[
\begin{ytableau}
4 \\ 
\none & 3 & 4 \\
\none & \none & 3 \\ 
\none & \none & \none & 1 \\ 
\end{ytableau}
\]
\end{example}

\begin{thm}\label{gt2ssyt}
There is a one-to-one correspondence between straight-shaped SSYT of shape $\lambda = (\lambda_1, \lambda_2, \dots, \lambda_d)$ and weight vector $w= (w_1, w_2, \dots, w_n) $ and triangle-shaped GT patterns with top row $\lambda$ and the same weight vector.
\end{thm}

This is a well-known fact (see, e.g., \cite{Stanleycombinatorics}), but we will prove it here, due to this bijection being used extensively throughout the paper. 

\begin{proof}
We first comment that the weight vector is counting different quantities of the SSYT versus the triangle GT pattern. Part of the proof is showing that these quantities are related through the bijection. To better explain each step of the proof, we consider the following explicit straight shape SSYT and triangle GT pattern which will be equivalent under the isomorphism:
\[
\begin{ytableau}
4 \\ 
3 & 3 & 4 \\
2 & 2 & 3 \\ 
1 & 1 & 1 & 1 \\ 
\end{ytableau} \quad\qquad \raisebox{-15pt}{\textrm{and}} \qquad\qquad \raisebox{-15pt}{$\begin{matrix}
4 & & 3 & & 3 & & 1 \\
& 4 & & 3 & & 2 & \\ 
& & 4 & & 2 & & \\ 
& & & 4 & & & 
\end{matrix}$.}
\]
Note that by counting the number of times each element appears in the SSYT, we see that it has weight $(4,2,3,2)$. By taking differences of row sums, we see that the GT pattern has weight $(4,2,3,2)$.  Further, note that the shape of the SSYT is $(4,3,3,1)$, which is precisely the top row of the GT pattern.

We begin by proving that a straight shape SSYT $T$ with shape $\lambda = (\lambda_1, \lambda_2, \dots, \lambda_d)$ and weight vector $w= (w_1, w_2, \dots, w_n) $ yields a triangle GT pattern with the same weight vector. The key idea is to iteratively build rows of the GT pattern, where the entries of each row will be the shape of a Young diagram determined by some of the entries of the SSYT.

Let the partition $\lambda_1:= (\lambda_{1,1})$ be associated with the $w_1$ boxes containing $1$, i.e., $\lambda_{1,1}=w_1$. This partition will contain at most 1 part because the tableau is semistandard and a $1$ cannot be placed directly above another $1$. 

In our explicit example, this step corresponds to
\[
\begin{ytableau} 
1 & 1 & 1 & 1 \\ 
\end{ytableau} \quad\qquad \raisebox{5pt}{$\Rar$} \qquad \raisebox{5pt}{$\begin{matrix}
& & & 4 & & & 
\end{matrix}$.}
\]

Now consider the partition $\lambda_2 := (\lambda_{2,1}, \lambda_{2,2})$ associated with the $w_1 + w_2$ boxes containing $1$ and $2$. This partition will contain at most 2 parts for the same reason as above. Since $\lambda_2$ contains both boxes filled with entries $1$ and $2$, we necessarily have that $\lambda_{2,1} \geq \lambda_{1,1} \geq \lambda_{2,2}$. This is because $\lambda_{1,1}$ can be extended by $2$'s, i.e. $\lambda_{2,1} \geq \lambda_{1,1}$. Additionally, a $2$ can be placed above a $1$, but not another $2$ yielding $\lambda_{1,1} \geq \lambda_{2,2}$. 

In explicit example, we can also observe that $4 \geq 4 \geq 2$, since the bottom row in the SSYT does not contain any 2's. 
\[
\begin{ytableau}
2 & 2   \\ 
1 & 1 & 1 & 1 \\ 
\end{ytableau} \quad\qquad \raisebox{-5pt}{$\Rar$} \qquad\qquad \raisebox{-5pt}{$\begin{matrix}
& & 4 & & 2 & & \\ 
& & & 4 & & & 
\end{matrix}$.}
\]

We continue this process where $\lambda_i := (\lambda_{i,1}, \dots \lambda_{i,i})$ is the partition that contains the entries $\{1, 2, \dots, i\}$. Due to the semistandard condition, we are guaranteed that $\lambda_i$ has at most $i$ parts. Additionally, since each $i$ can extend a row containing at most $w_{i-1}$ $i-1$'s but not be placed above another $i$, we have that for $j \leq i$
\[
\lambda_{i,j} \geq \lambda_{i-1,j} \geq \lambda_{i,j+1}.
\]
Clearly, since each $\lambda_i$ is a partition, we necessarily have that $\lambda_{i,j} \geq \lambda_{i,j+1}$ by definition. Thus, we have a GT pattern as desired. 

This is best illustrated in the next step of our explicit example. We can define $\lambda_3 = (4,3,2)$ as the partition that contains all entries $\leq 3$: the first row of 4 boxes contains only 1's, the second row contains 2's and 3's and in the third row, we only consider the first 2 boxes containing a 3. Observe that we have $4 \geq 4 \geq 3 \geq 2 \geq 2$. Since $\lambda_3$ is a partition and is weakly decreasing by definition, we necessarily have $4 \geq 3 \geq 2$. The same is true for $\lambda_2$ as well. Additionally, since the first row is not extended by anything, but the second row is extended by a 3, we have $4 \geq 4 \geq 3 \geq 2 \geq 2$ as desired.
\[
\begin{ytableau} 
3 & 3  \\
2 & 2 & 3 \\ 
1 & 1 & 1 & 1 \\ 
\end{ytableau} \quad\qquad \raisebox{-10pt}{$\Rar$} \qquad\qquad \raisebox{-10pt}{$\begin{matrix}
& 4 & & 3 & & 2 & \\ 
& & 4 & & 2 & & \\ 
& & & 4 & & & 
\end{matrix}$.}
\]
The remaining step will yield the original SSYT and GT pattern as the shape of the SSYT is $(4,3,3,1)$, which is precisely the top row of the GT pattern.

Now we prove that a triangle GT pattern yields a straight shape SSYT. The key idea will be to iteratively build a larger and larger Young diagram by appending skew diagrams based on the shape parameters from each row of the GT pattern.  When the Young diagram is expanded, the new boxes will be filled with copies of the iteration number.

We first fill a Young diagram of shape $(\lambda_{1,j})_{j=1}^1$ with $1$'s. In other words, this becomes the SSYT of shape $(\lambda_{1,j})_{j=1}^1$ with weight vector $(w_1)$. This is guaranteed to be semistandard since there is at most 1 part to the partition, which yields a weakly increasing row. Looking at our explicit example, we have
\[
\raisebox{5pt}{$\begin{matrix}
& & & 4 & & & 
\end{matrix}$}
\quad\qquad \raisebox{5pt}{$\Rar$} \qquad \begin{ytableau} 
1 & 1 & 1 & 1 \\ 
\end{ytableau}\enskip.
\]

Next, we fill the skew Young diagram $(\lambda_{2,j})_{j=1}^2 / (\lambda_{1,j})_{j=1}^1$ with $2$'s and chain it with the SSYT $(\lambda_{1,j})_{j=1}^1$ filled with $1$'s. In other words, concatenate the two tableaux together like a puzzle piece. This chain is guaranteed to be semistandard since $(\lambda_{2,j})_{j=1}^2 / (\lambda_{1,j})_{j=1}^1$ will extend each $\lambda_{1,j}$ by $2$'s and thus the row will remain weakly increasing. Additionally, since the skew tableau is semistandard, we are guaranteed that the columns will remain strictly increasing. Notice that this implies that $(\lambda_{2,j})_{j=1}^2$ has weight $(w_1, w_2)$.

In the explicit example, we have that the skew shape created by the first 2 rows is $(4,2)/(4)$. This is the skew shape with 2 boxes in the 2nd row. Chaining this together with the tableaux of shape $(4)$ yields a SSYT of shape $(4,2)$ with all 1's ones in the first row and all 2's in the second row. 
\[\raisebox{-5pt}{$\begin{matrix}
& & 4 & & 2 & & \\ 
& & & 4 & & & 
\end{matrix}$}
 \quad\qquad \raisebox{-5pt}{$\Rar$} \qquad\qquad 
 \begin{ytableau}
{} & {}   \\ 
\none & \none & \none & \none \\ 
\end{ytableau}\quad\qquad \raisebox{-5pt}{$\Rar$} \qquad\qquad  \begin{ytableau}
2 & 2   \\ 
1 & 1 & 1 & 1 \\ 
\end{ytableau} \enskip.
\]

We continue this process such that the skewed Young diagram $(\lambda_{i+1,j})_{j=1}^{i+1} / (\lambda_{i,j})_{j=1}^i$ with $w_{i+1}$'s chained with previous diagram yields a tableau of shape  $(\lambda_{i+1,j})_{j=1}^{i+1}$ with weight vector $(w_1,  \dots w_{i+1})$. The resulting tableaux will be clearly semistandard because the interlacing inequalities of the GT pattern ensure that none of the skew shape have a vertical strip greater than 1 box high, and the concatenation process does not violate the weakly increasing property of the rows. Thus, we have an SSYT as desired.

Looking at the remaining two iterations applied to the explicit example, we have
\[
\raisebox{-10pt}{$\begin{matrix}
& 4 & & 3 & & 2 & \\ 
& & 4 & & 2 & & \\ 
& & & 4 & & & 
\end{matrix}$}\quad\qquad \raisebox{-10pt}{$\Rar$} \qquad\qquad
\begin{ytableau} 
{} & {}  \\
\none & \none & {} \\ 
\none & \none & \none & \none \\ 
\end{ytableau} \quad\qquad \raisebox{-10pt}{$\Rar$} \qquad\qquad \begin{ytableau} 
3 & 3  \\
2 & 2 & 3 \\ 
1 & 1 & 1 & 1 \\ 
\end{ytableau}
\]
and
\[
\raisebox{-15pt}{$\begin{matrix}
4 & & 3 & & 3 & & 1 \\
& 4 & & 3 & & 2 & \\ 
& & 4 & & 2 & & \\ 
& & & 4 & & & 
\end{matrix}$}\quad\qquad \raisebox{-15pt}{$\Rar$} \qquad\qquad
\begin{ytableau} 
{}  \\
\none & \none & {}  \\
\none & \none & \none \\ 
\none & \none & \none & \none \\ 
\end{ytableau} \quad\qquad \raisebox{-15pt}{$\Rar$} \qquad\qquad \begin{ytableau} 
4 \\
3 & 3 & 4 \\
2 & 2 & 3 \\ 
1 & 1 & 1 & 1 \\ 
\end{ytableau} \enskip \raisebox{-15pt}{.}
\]

Notice in this process that the skew shape formed by $(\lambda_{i+1,j})_{j=1}^{i+1} / (\lambda_{i,j})_{j=1}^i$ and filled with $w_{i+1}$ $i+1$'s is another way to represent the weight of a GT pattern. Skewing by adjacent rows and filling with a specified content is exactly taking the difference between those two adjacent rows. This yields the weight vector being the same in both the GT pattern and SSYT.
\end{proof}

This proof serves as an algorithm for transitioning between integer GT patterns and SSYT for the remainder of the paper.

\begin{thm}\label{gt2ssyt2}
There is a one-to-one correspondence between skew shape SSYT of shape $\lambda / \mu$ for some integer partitions $\lambda$ and $\mu$ and parallelogram GT patterns with top row $\lambda$ and bottom row $\mu$ and the same weight vector. 
\end{thm}

\begin{proof}
The proof is the same as above except when looking at the Young diagram $(\lambda_{i+1,j})_{j=1}^2 / (\lambda_{1,j})_{j=1}^1 $, the resulting shape is filled with $i$'s instead of $i+1$'s.
\end{proof}

\begin{example}\label{exGT2SSYT}

Consider the following parallelogram GT pattern
\[
\begin{matrix}
4 & & 3 & & 2 & & 1 & & & &\\
& 4 & & 2 & & 2 & & 1 & & & \\
& & 2 & & 2 & & 1 & & 0 & &\\
& & & 2 & & 2 & & 1 & & 0 &\\ 
& & & & 2 & & 1 & & 1 & & 0 \\
\end{matrix}
\]
that corresponds to the skew shape SSYT
\[
\begin{ytableau}
3 \\
\none & 3 \\ 
\none & 1 & 4 \\
\none & \none & 3 & 3 \\
\end{ytableau}
\]
The process is essentially the same for the skew SSYT as it is for the straight shape one. However, instead of looking at the $i$th and $(i-1)$th row in the GT pattern, you look at the $(i+1)$th and $i$th row instead. Further notice that there are no $2$ entries in the skew SSYT. This corresponds to the 2nd and 3rd row of the GT pattern being equal. Thus, when you skew the 3rd row by the 2nd row, you end up with the empty partition. 
\end{example}

\subsection{Eigensteps}\label{eigensteps}
The eigensteps of a frame allow us to tie together the combinatorial objects introduced in the previous section and frame theory and may be viewed as a generalization of the Schur-Horn theorem \cite{Schurbackground, Hornbackground} relating existence of a self-adjoint matrix to interlacing inequalities between the spectrum and the diagonal elements. 

Recall that for a $(d,n)$-frame $\Phi = [ \phi_1, \phi_2, \dots, \phi_n]$, we have both the frame operator $\frameop$ and the Gram matrix $\gram$. From this we introduce the notion of partial sequences of vectors: $\Phi_i := ( \phi_{j})^i_{j=1}$. In other words for every $i\in [n]$, the $i$th partial sequence of a frame is the first $i$ vectors or the first $i$ columns of the synthesis matrix. From these partial sequences, it still makes sense to talk about the associated frame operator and Gram matrix.  We will use double indexing sequences $(\lambda_{i,j})_{j=1}^d$ to denote the spectra of the $i$th partial sum of the frame operator. Note that Lemma~\ref{lem:sumnorm} holds for every partial sequence. Thus for any $i \in [n]$ we have:
\begin{equation}\label{eqn:parsums}
    \sum_{j=1}^d \lambda_{i,j} = \sum_{i=1}^n \norm{\phi_i}^2
\end{equation}
Further notice that as $i$ increases (i.e., as we add more vectors to each partial sum), $\Phi_i\Phi_i^\ast$ remains a $d \times d$ matrix while $\Phi_i^\ast\Phi_i$ increases in size with each vector added. This is because the number of rows in each partial sequence remains the same, but as $i$ increases, then so do the number of columns of $\Phi_i$.  Thus, the spectra of the Gram matrices of the partial sequence are indexed as $(\lambda_{i,j})_{j=1}^i$ but also satisfy~\eqref{eqn:parsums}.

To introduce further terminology from \cite{cfmconstructingfiniteframes}, we say that one sequence $( \gamma_i)^n_{i=1}$ \emph{interlaces} another sequence $( \beta_i)^n_{i=1}$ if $$ \beta_n \leq \gamma_n \leq \beta_{n-1} \leq \gamma_{n-1} \leq \dots \leq \beta_1 \leq \gamma_1.$$ In fact, the sequences do not need to be the same size (cf.\ \cite{cfpconstructingselfadjoint}): a sequence $ ( \alpha_j)^{n-1}_{j=1}$ interlaces a sequence $( \nu_j)^n_{j=1}$ if $\nu_{j+1} \leq \alpha_j \leq \nu_j$ for all $j= 1, \dots, n-1$. We write $(\beta_i)^n_{i=1} \sqsubseteq (\gamma_i)^n_{i=1}$ and $(\alpha_j)^{n-1}_{j=1} \sqsubseteq (\nu_j)^n_{j=1}$, respectively.

Utilizing the above notation, we can introduce the following definitions for the sequence of eigensteps of a frame~\cite{cfmconstructingfiniteframes,cfpconstructingselfadjoint}.

\begin{definition}\label{outerdef}
Let $(\lambda_j)^d_{j=1}$ and $(w_k)^n_{k=1}$ be nonnegative, nonincreasing sequences. Corresponding \emph{outer eigensteps} are a sequence of sequences $((\lambda_{i,j})^d_{j=1})^n_{i=1}$ which satisfies the following properties: 
\begin{enumerate}
\item $\lambda_{0,j}=0$ for every $j\in [d]$,
\item $\lambda_{n,j}=\lambda_j$ for every $j\in [d]$,
\item $(\lambda_{i-1,j})^d_{j=1} \sqsubseteq ( \lambda_{i,j})^d_{j=1}$ for every $i \in [n]$, and
\item $\sum^d_{j=1} \lambda_{i,j} = \sum^i_{k=1}w_k$ for every $i \in [n]$. 
\end{enumerate}
\end{definition}

Leveraging Schur-Horn~\cite{Schurbackground, Hornbackground}, it is shown in \cite{cfmconstructingfiniteframes} that given any $\Phi = (\phi_i)^n_{i=1}$ with $\frameop$ that has spectrum $(\lambda_j)^d_{j=1}$ and squared norms $ w_i = ||\phi_i||^2$ for all $i$ will generate a sequence of outer eigensteps as sequences of spectra associated with the partial sums of the frame operator. Notice that the fourth condition preserves the trace condition for each partial sum. In addition to looking at the partial sum of $\frameop$, we can equivalently look at $\gram$ as well.

\begin{definition}\label{innerdef}
Let $(\lambda_j)^n_{j=1}$ and $(w_k)^n_{k=1}$ be nonnegative, nonincreasing sequences. Corresponding \emph{inner eigensteps} are a sequence of sequences $((\lambda_{i,j})^n_{j=1})^n_{i=1}$ which satisfies the following properties: 
\begin{enumerate}
\item $\lambda_{i,j}=\lambda_j$ for every $j\in [d]$,
\item $(\lambda_{i-1,j})^{i-1}_{j=1} \sqsubseteq (\lambda_{i,j})^i_{j=1}$ for every $i = 2, \dots, n$,
\item $\sum^d_{j=1} \lambda_{i,j} =  \sum^i_{k=1}w_k$ for every $i\in [n]$, and
\item $\lambda_{i,j} = 0$ for $j > d$.
\end{enumerate}
\end{definition}

Notice that in this definition, the interlacing condition involves sequences of different lengths. Similar to above, work from \cite{cfpconstructingselfadjoint} proves that the partial sequence of $\gram$ generates a sequence of inner eigensteps. The names outer and inner eignesteps originate from the difference in using the outer product to generate the partial sums of the frame operator versus the partial sums of the Gram matrix being comprised of the inner products between the partial sequences of frame vectors.

\begin{prop}
The outer eigensteps of a frame yields a unique set of inner eigensteps. Similarly, the converse is true: the inner eigensteps of a frame yields a unique set of outer eigensteps.
\end{prop}

The inner eigensteps of a frame always form a not-necessarily-integer triangular GT pattern, while the outer eigensteps yield a not-necessarily integer parallelogram GT pattern.

Eigensteps are an important characterization arising from the Schur-Horn theorem; a self-adjoint matrix exists if and only if the eigensteps satisfy the interlacing inequalities defined above \cite{Schurbackground, Hornbackground}. The next question to ask after such a matrix exists is whether it can be constructed. Theorems 2 \& 7 from \cite{cfmconstructingfiniteframes} provide algorithms for constructing a finite frame based on a sequences of outer eigensteps (Theorem 7 can be seen as a revised version to Theorem 2 that is easier to implement). From the proposition above, it is easy to go between inner and outer eigensteps, simply pad (or take away) by the appropriate number of 0's. Thus, if given an allowable sequence of eigensteps that fits either definition above, it is possible to construct a finite frame. 

However, the question still remains on how to construct an allowable eigenstep pattern. One way to construct such a sequence of eigensteps would be to leverage the interlacing inequalities in order to algebraically determine a correct sequence. Constructing such a sequence is also a combinatorial problem, and we can leverage well-studied combinatorial objects, such as GT patterns and SSYTs, in order to construct allowable sequences. In the next section, we describe the correspondence between these two combinatorial objects and the sequences of eigensteps.

\section{Clearing Frames}\label{semistandardframes}
In order to use the machinery of SSYTs to construct and analyze frames, we need to classify frames that yield integer GT patterns.  This leads us to the following definition.

\begin{definition}
We call a frame $\Phi = (\phi_i)^n_{i=1}$ a \emph{clearable frame} if its outer eigensteps (and consequently inner eigensteps) can be scaled by a single number $\ell$ such that the sequences of eigensteps are nonnegative integers. In this case $\ell$ is called a \emph{clearing constant} and $(\sqrt{\ell}\phi_i)^n_{i=1}$ a \emph{cleared frame}.
\end{definition}

\begin{prop}\label{prop:clear2GT}
The eigensteps of a clearable frame $\Phi$ are integer GT patterns. The outer eigensteps of $\Phi$ correspond to parallelogram GT patterns, and the inner eigensteps correspond to triangular GT patterns.
\end{prop}

\begin{proof}
First notice that for a sequence of normalized eigensteps, $((\lambda_{i,j})^i_{j=1})^n_{i=1}$, each of the $(\lambda_{i,j})^i_{j=1}$ is nonnegative and nonincreasing by definition. This is exactly the first criteria of Definition~\ref{GTdef}. Additionally, condition (ii) in Definition~\ref{GTdef} is precisely the interlacing condition for outer and inner eigensteps in Definition~\ref{outerdef} and Definition~\ref{innerdef}. Notice that since the length of the sequences in the outer eigensteps are the same length, they will are parallelogram GT patterns, while the inner eigensteps will be triangular GT patterns since the sequences of inner eigensteps increase in length.
\end{proof}

\begin{example}
This is an example built upon one in \cite{cfpconstructingselfadjoint}. Consider the finite unit-norm tight frame below:
\[
\Phi = \begin{bmatrix}
1 & \frac{2}{3} & -\frac{1}{\sqrt{6}} & -\frac{1}{6} & \frac{1}{6} \\
0 & \frac{\sqrt{5}}{3} & \frac{\sqrt{5}}{\sqrt{6}} & \frac{\sqrt{5}}{\sqrt{6}} & -\frac{\sqrt{5}}{\sqrt{6}} \\
0 &  0 & 0  & \frac{\sqrt{5}}{\sqrt{6}}  & \frac{\sqrt{5}}{\sqrt{6}}
\end{bmatrix},
\]
which has corresponding outer and inner eigensteps
\begin{align*}
\begin{matrix}
\frac{5}{3} & & \frac{5}{3} & & \frac{5}{3} \\
& \frac{5}{3} & & \frac{5}{3} & & \frac{2}{3} \\
& & \frac{5}{3} & & \frac{4}{3} & & 0 \\ 
& & & \frac{5}{3} & & \frac{1}{3} & & 0 \\ 
& & & & 1 & & 0 & & 0 \\ 
\end{matrix}
&&
\begin{matrix}
\frac{5}{3} & & \frac{5}{3} & & \frac{5}{3} & & 0 & & 0  \\
& \frac{5}{3} & & \frac{5}{3} & & \frac{2}{3} & & 0  \\
& & \frac{5}{3} & & \frac{4}{3} & & 0 \\ 
& & & \frac{5}{3} & & \frac{1}{3} \\ 
& & & & 1 \\ 
\end{matrix},
\end{align*}
which in turn can be cleared to integer GT patterns:
\begin{align*}
\begin{matrix}
5 & & 5 & & 5 \\
& 5 & & 5 & & 2 \\
& & 5 & & 4 & & 0 \\ 
& & & 5 & & 1 & & 0 \\ 
& & & & 3 & & 0 & & 0 \\ 
\end{matrix}
&&
\begin{matrix}
5 & & 5 & & 5 & & 0 & & 0  \\
& 5 & & 5 & & 2 & & 0  \\
& & 5 & & 4 & & 0 \\ 
& & & 5 & & 1 & & \\ 
& & & & 3 & \\ 
\end{matrix}
\end{align*}
Notice that the value that clears the eigensteps of this frame is 3, which is precisely the same as the dimension of $\Phi$. 
\end{example}
It is a fact that certain collections of (not necessarily integral) Gelfand-Tsetlin patterns form polytopes.  The vertices of these polytopes are clearable \cite{deloera2004vertices}.

\begin{corollary}\label{eig2ssyt}
The outer eigensteps of a clearable frame correspond to skew SSYT, and the inner eigensteps correspond to straight shape SSYT
\end{corollary}

\begin{proof}
This is trivial due to Theorems~\ref{gt2ssyt} and~\ref{gt2ssyt2}.
\end{proof}

\begin{example}
Continuing the example from above, we have that the corresponding straight shape and skew shape tableaux are: 
\begin{align*}
\begin{matrix}
5 & & 5 & & 5 \\
& 5 & & 5 & & 2 \\
& & 5 & & 4 & & 0 \\ 
& & & 5 & & 1 & & 0 \\ 
& & & & 3 & & 0 & & 0 \\ 
\end{matrix}
&& \raisebox{5pt}{$\longleftrightarrow$} &&
\raisebox{15pt}{$\begin{ytableau}
3 & 3 & 4 & 4 & 4 \\
1 & 2 & 2 & 2 & 3 \\
\none & \none & \none & 1 & 1 
\end{ytableau}$}
\end{align*}
\\
\begin{align*}
\begin{matrix}
5 & & 5 & & 5 & & 0 & & 0  \\
& 5 & & 5 & & 2 & & 0  \\
& & 5 & & 4 & & 0 \\ 
& & & 5 & & 1 & & \\ 
& & & & 3 & \\ 
\end{matrix}
&&\longleftrightarrow&& 
\raisebox{10pt}{$\begin{ytableau}
4 & 4 & 5 & 5 & 5 \\
2 & 3 & 3 & 3 & 4 \\
1 & 1 & 1 & 2 & 2 
\end{ytableau}$}
\end{align*}
\end{example}

Since the spectra of the frame operator is a zero-padded version of the spectra of the Gram, we can think of the corresponding parallelogram GT pattern as a ``zero-padded" version of the respective triangle GT pattern. Notice that because the inner and outer eigensteps are associated with the same frame this induces a rather trivial bijection between straight shape SSYT and skew shape SSYT. Although the following result is trivial as a result on SSYT, it is the first example in this paper of result solely concerning SSYT having an interpretation in the frame theory.

\begin{corollary}\label{straight2skew}
A straight SSYT with shape $\lambda = (\lambda_1, \lambda_2, \dots, \lambda_d)$ and filled with $[n]$ and weights $(w_1, w_2, \dots, w_n)$ is in bijective correspondence with a skew SSYT with shape $\lambda / (\lambda_{1,1})$, where $\lambda_{1,1}$ is the associated partition in the straight SSYT of boxes filled with only $1$, with entries $[n-1]$ and weights $(w_2, w_3,\dots, w_n)$.
\end{corollary}

\begin{proof}
For the forward direction, note that removing the boxes containing $1$'s in the tableau obviously produces a skew tableau of shape $\lambda / \lambda_{1,1}$, where $\lambda_{1,1}$ is the shape of the partition of boxes containing the entry $1$. Relabeling the remaining entries by $i - 1$ yields the desired tableau. For the backwards direction, add 1 to all the entries of the tableaux. Fill the remaining inner corner with boxes with $1$'s. The following tableau is guaranteed to be semistandard, since the partition's empty space only had a $1$'s part and thus the columns will remain strictly increasing and the rows weakly.
\end{proof}

Just as we would like to gather the eigensteps from the frame, we would like to be able to go the other way as well, that is, construct a frame based on a sequence of eigensteps. There are two methods for doing so proposed in \cite{cfmconstructingfiniteframes}. Theorems 2 \& 7 from \cite{cfmconstructingfiniteframes} provide algorithms for constructing a finite frame based on a sequences of outer eigensteps. 

In \cite{cfmconstructingfiniteframes}, the difficulty of finding a valid sequences of eigensteps is discussed. They attempt to remedy this issue with the Top Kill algorithm. However, in the case of clearable frames, we propose an alternative algorithm leveraging \ref{eig2ssyt}. Thus, by looking at ways to fill certain SSYT, we can appropriately choose valid sequences of eigensteps that will correspond to frames.

\begin{thm}\label{bigguy} 
An SSYT of shape $\lambda = (\lambda_1, \lambda_2, \dots, \lambda_d)$, weights $w = (w_1, w_2, \dots, w_n)$, and specified filling along with a clearing constant $\ell$ corresponds to a clearable frame's ($\Phi$) outer eigensteps such that the eigenvalues of $\Phi\Phi^\ast$ are $\frac{\lambda_1}{\ell}, \frac{\lambda_2}{\ell}, \dots \frac{\lambda_d}{\ell}$  with the norm-squareds of the frame vectors being $\frac{w_1}{\ell}, \frac{w_2}{\ell}, \dots, \frac{w_n}{\ell} $. 
\end{thm}

\begin{proof}
From Theorem~\ref{eig2ssyt} we know that we can associate the inner and outer eigensteps of a frame to a straight shape SSYT. We can construct a sequence of $(w_i)^n_{i=1} = w $. By constructing $ (w_i)^n_{i=1}$ in this way, we automatically satisfy the third condition of Definition~\ref{innerdef}, since a tableau and GT pattern have the same weight vector. Thus we can apply Theorem 7 in \cite{cfmconstructingfiniteframes} in order to construct a clearable frame, by dividing all entries in the GT pattern by the ``clearing constant" $\ell$. \\

Going backwards now, we can take a clearable frame and find its outer eigensteps. Since the frame is clearable, we know there exists some constant $\ell$ that turns the eigensteps into an integer GT pattern. Thus, we can construct the associated straight shape SSYT according to Theorems~\ref{gt2ssyt} and~\ref{gt2ssyt2}. Since these correspondences are unique, the frame is then associated with a unique straight shape SSYT with respect to the constant $\ell$.  
\end{proof}

\begin{cor}\label{bigguy2} 
A skew SSYT with shape $\lambda / (w_1)$, where $\lambda = (\lambda_1, \lambda_2, \dots,\lambda_d)$, with weights $w=(w_2, w_3, \dots, w_n)$, and specified filling along with a clearing constant $\ell$ corresponds to a clearable frame's ($\Phi$) inner eigensteps such that the eigenvalues of $\Phi\Phi^\ast$ are $\frac{\lambda_1}{\ell}, \frac{\lambda_2}{\ell}, \dots \frac{\lambda_d}{\ell}$  with the norm-squareds of the frame vectors being $\frac{w_1}{\ell}, \frac{w_2}{\ell}, \dots, \frac{w_n}{\ell} $. 
\end{cor}
\begin{proof}
This follows immediately from Theorem~\ref{bigguy} and Corollary~\ref{straight2skew}.
\end{proof}

However, once we narrow down our scope into specific frames, this process gets a little harder. When we place additional conditions on what a frame should look like, we need to do the same with the tableaux. \\ 

The Top Kill Algorithm outlined in \cite{cfpconstructingselfadjoint} provides an algorithm for constructing valid sequences of eigensteps for finite unit norm tight frames (FUNTFs). The work below combinatorially analyzes how to choose valid sequences of eigensteps for clearable frames through analyzing specific conditions on tableaux. One could reinterpret the steps of Top Kill as iteratively generating skew SSYT to fill a tableaux.
We can expand off of Theorem~\ref{bigguy} for finite equal norm frames and FUNTFs. However, we can first describe some properties about the correspondence above. Many of these properties trivially follow from the correspondence, but it is useful to record them all. 

\begin{corollary}\label{cor:SSYTprops}
The following properties about the correspondence from Theorem~\ref{bigguy} between a straight shape SSYT $T$ and cleared frame $\Phi$ hold: 
\begin{enumerate}
\item The number of boxes in the longest row of $T$ is the optimal upper frame bound.  The number of boxes in the shortest row of $T$ is the optimal lower frame bound.
\item The shape of the SSYT is the spectrum of $\frameop$.
\item The number of rows in $T$ is the rank of $\Phi$.
\item The weight of $T$ is the sequence of norm-squares of the vectors.
\item The maximum entry in $T$ is the number of frame vectors in $\Phi$.
\end{enumerate}
\end{corollary}

\begin{proof} 
To prove 1. \& 2., note that the top row of the corresponding triangular GT pattern is the spectrum of $\gram$, where the non-zero entries both yield the shape of $T$ and the spectrum of $\frameop$. For 3., note that the number of non-zero entries of the top row of the GT pattern both yields the number of rows in the SSYT and $\rk(\gram)=\rk(\Phi)$. To show 4., we see that the $i$th weight $w_i$ of $T$ is exactly the difference between the $i$th and $(i-1)$th row sum of the GT pattern, where the $i$th row sum is the sum of the spectrum of the Gram matrix of the first $i$ vectors, i.e., (as in Lemma~\ref{lem:sumnorm}) the sum of the norm-squares of the first $i$ vectors.  Thus, $w_i = \norm{\varphi_i}^2$.
 Finally, to prove 5., note that the maximum entry in $T$ is determined by the number of rows that the corresponding GT pattern has, which is determined by the number of partial sums of $\Phi$. There are as many partial sums as there are frame vectors. Thus, the number of rows of the GT pattern is determined by the amount of frame vectors. 
\end{proof}

As FUNTFs and more generally equal-norm tight frames are desirable, we have the following theorem.

\begin{thm}\label{ssyt2funtf}
Rectangular SSYTs correspond to clearable tight frames. SSYTs with constant weight correspond to clearable equal-norm frames.  
A straight shape SSYT with shape $\lambda  = \underbrace{(\frac{\ell n}{d}, \dots, \frac{\ell n}{d})}_{d \text{ copies}}$, (i.e., the SSYT is a $d \times \ell n $ rectangle), with weights $\underbrace{(\ell, \dots, \ell)}_{n \text{ copies}}$ corresponds to a clearable finite unit norm tight frame of $n$ vectors in $\bF^d$. 
\end{thm}

\begin{proof}
  Recall that for a frame $\Phi$, tightness is equivalent to saying that $\Phi\Phi^\ast = A I$, for some frame bound $A$.  If the frame is cleared, then $A \in \bN$, and Corollary~\ref{cor:SSYTprops} tells us that the shape of the SSYT is $(A,A,\dots, A)$ of length $d$.  The second statement follows immediately from Corollary~\ref{cor:SSYTprops}.3.

We now consider a FUNTF $\Phi=(\varphi_i)_{i=1}^n$ of $n$ vectors in $\bF^d$, which by Lemma~\ref{lem:sumnorm} have norm-squares $(1,1,\dots,1)$ and frame bound $n/d$. So, if $\Phi$ is clearable with clearing constant $\ell$, $(\sqrt{\ell}\varphi_i)_{i=1}^n$ will have constant norm-square $\ell$ and frame bound $\ell n/d$, yielding the desired result,
\end{proof}

Note that we can always rescale any FUNTF to be an equal-norm tight frame with either a chosen frame bound or a chosen norm.  Thus, Theorem~\ref{ssyt2funtf} applies to any clearable equal-norm tight frame. Using $\ell =d$ as the clearing constant, resulting in an equal-norm tight frame with vector norm square $d$ and frame bound $n$, is actually somewhat natural, as it can be viewed as the scaling for so-called real equiangular tight frames associated with Seidel adjacency matrices~\cite{VaSei66} and for Fourier frames without rescaling. This scaling was also used in~\cite{HPpolytope}, where polytope geometry was leveraged to better understand eigensteps.

Notice that the above theorems are stated using the inner eigensteps of a frame and consequently the straight shape SSYT. However, due to Corollary~\ref{straight2skew} these theorems can also be stated through the outer eigensteps and the associated skew-shaped SSYT. It is often easier to look at the straight shape SSYT rather than the skew shape, which is why the focus was placed on the inner eigensteps.

\subsection{Naimark complements}\label{naimarkcomplements}

Naimark complementation is a sort of duality applied to frames which is a useful tool in constructing frame with certain properties or reducing a problem to an easier one  \cite{holmes2004optimal,casazza2013kadison,CFMfusion}.

\begin{definition}\label{defn:naimarktight}
Let $\Phi$ be a tight frame with frame bound $A$. Consider a set $\Psi = ( \psi_i)^{n}_{i=1}$ such that $\gram+\Psi^\ast\Psi = AI$. $\Psi$ is called the \emph{Naimark complement} of $\Phi$.
\end{definition}

If $\Psi$ is a Naimark complement of an $A$-tight $(d,n)$-frame, then $\Psi^\ast \Psi = AI - \Phi^\ast \Phi$ implies that $\Psi$ is a tight frame of $n$ vectors spanning an $(n-d)$-dimensional space, which may be chosen after choice of basis to be $\bF^{n-d}$.  Further, it immediately follows from the definition that the Naimark complement of an equal norm frame is equal norm. Other frame properties are either preserved or dualized in a predictable way when taking the Naimark complement, leading to the usefulness of Naimark complementation as a tool. Thus, it makes sense to talk about the eigensteps of the Naimark complement of a frame as well.  Although a Naimark complement $\Psi$ is not unique, the Gram matrix of a Naimark complement is uniquely defined, meaning that the eigensteps are as well.

All possible eigensteps for $(d,n)$-frames with frame operator spectrum $\lambda = (\lambda_1, \lambda_2, \dots, \lambda_n)$ form a polytope (see, e.g., \cite{cahill2017connectivity}). In \cite{HPpolytope}, the polytope geometry of eigensteps, in particular of equal-norm $(d,n)$ tight frames with frame bound $n$ and squared norms $d$, was leveraged to further characterize eigensteps.
\begin{definition}
    For integers $d < n$, let $\Lambda_{n,d}$ denote the polytope of all triangular GT patterns (not necessarily integral) with top row $\lambda = (\,\,\underbrace{n, \dots, n}_{d \text{ copies}}, \underbrace{0, \dots, 0}_{n-d \text{ copies}})$ and weights $w=\underbrace{(d, \dots, d)}_{n \text{ copies}}$.
\end{definition}
In \cite{HPpolytope}, an isomorphism between $\Lambda_{n,d}$ and $\Lambda_{n,n-d}$ was introduced which mapped eigensteps of a frame to eigensteps of the Naimark complement of the frame. Note that their indexing is relative to outer eigensteps, which we have translated to inner eigensteps.

\begin{prop}\cite{HPpolytope}\label{prop:Naim_polytope}
There exists an involutive isomorphism
\[
N_{n,d}: \Lambda_{n,d} \longrightarrow \Lambda_{n, n-d}
\]
given by 
\[
(N_{n,d}(\lambda))_{i,j}= \begin{cases}
\lambda_{n-i, d+j-i}, & \text{ for } j \leq i \leq d+j-1, \enskip j \leq n-d  \\
0, & \text{ for } j > n-d \\
n, & \text{ for } i > d+j-1,\enskip j \leq n-d \\
\end{cases}
\]
which maps eigensteps of a frame to the eigensteps of its Naimark complement.
\end{prop}
Note that this isomorphism is well-defined on GT patterns for all equal-norm tight frames, not just the clearable ones.  However, it immediately follows from the definition that cleared frames are mapped to cleared frames.

\begin{example}
For example, if $n=5$ and $d=3$ we have that
\begin{align*}
\begin{matrix}
5 & & 5 & & 5 & & 0 & & 0  \\
& 5 & & 5 & & 2 & & 0  \\
& & 5 & & 4 & & 0 \\ 
& & & 5 & & 1 & & \\ 
& & & & 3 & \\ 
\end{matrix}
&&\longleftrightarrow &&
\begin{matrix}
5 & & 5 & & 0 & & 0 & & 0 \\
&5 & & 3 & & 0 & & 0 \\
&&5 & & 1 & & 0 \\
&&&4 & & 0 \\ 
&&&& 2 
\end{matrix}
\end{align*}
Notice that this size of the triangles on the top consisting solely of $0$'s or of $n$'s are swapped, while the center parallelogram flips upside down. 
\end{example}

Alternatively, we can construct a bijection between tableaux that is equivalent to the above involution on integer GT patterns. This bijection is essentially a generalization of the Robinson–Schensted–Knuth (RSK) correspondence used in recent work to give a combinatorial proof of a geometric result concerning maps of curves to projective space~\cite{GRgeneralized}. This gives further evidence that active research on tableaux could be used to better understand frames.

\begin{thm}\label{tighttableaux}
Fix $d < n$ and denote 
\[
\lambda_n^d=\underbrace{(n, \dots, n)}_{d \text{ copies}} \quad \textrm{and} \quad w_d^n =\underbrace{(d, \dots, d)}_{n \text{ copies}}.
\]
Let $T$ be an SSYT of shape $\lambda_n^d$ with constant weights $w_d^n$. We define $\gamma(T)$ to be the SSYT constructed from the following algorithm: 
\begin{enumerate}
\item Construct an intermediary filling $\Gamma(T)$ of the Young diagram of $T$, which is not itself an SSYT, by replacing each entry $j$ in $T$ with $n+1-j$. 
\item Construct the $i$th column of $\gamma(T)$ by taking the set complement of the $i$th column of $\Gamma(T)$ from $[n]$ and then reordering to be strictly increasing from bottom to top.
\end{enumerate}
Then $\gamma$ is a bijection from SSYT of shape $\lambda_n^d$ and weights $w_d^n$ to SSYT of shape $\lambda_n^{n-d}$ and constant weights $w_{n-d}^n$
\end{thm}
We first give an example before proving the theorem.
\begin{example}
Consider the tableau below: 
\begin{align*}
\begin{ytableau}
4 & 4 & 5 & 5 & 5 \\
2 & 3 & 3 & 3 & 4 \\
1 & 1 & 1 & 2 & 2 
\end{ytableau} &&
\begin{ytableau}
\color{red}2 &\color{red} 2 &\color{red} 1 & \color{red}1 & \color{red}1 \\
\color{red}4 & \color{red}3 & \color{red}3 & \color{red}3 & \color{red}2 \\
\color{red}5 &\color{red} 5 & \color{red}5 & \color{red}4 & \color{red}4 
\end{ytableau} && 
\begin{ytableau}
3 & 4 & 4 & 5 &5\\
1 & 1 & 2 & 2 &3
\end{ytableau}
\end{align*}
If we start with the leftmost SSYT $T$ and replace each entry $j$ with $n+1-j$, we obtain the Young diagram filling $\Gamma(T)$ with red entries. Further, if we then look at the red tableau and take the order preserving set complement of each column we obtain the rightmost SSYT $\gamma(T)$. We call the leftmost and rightmost SSYT \emph{complements} of each other. 
\end{example}

\begin{proof}
We want to prove that $\gamma$ is in fact a bijection between SSYT filled with $[n]$ with equal weight $d$ and SSYT filled with $[n]$ with equal weight $n-d$. 

Let $T$ be an SSYT filled with $[n]$ with equal weight $d$. First note that $\Gamma(T)$ has $d$ copies of each element of $[n]$ by construction.  Then, when taking the set complements of the columns to go from $\Gamma(T)$ to $\gamma(T)$, we get that each element of $[n]$ appears exactly $n-d$ times.

We now show that $\gamma(T)$ is an SSYT. Notice that Step 1 in the above algorithm reverses the semistandard-ness of the SSYT: the rows of $\Gamma(T)$ are weakly decreasing and the columns are strongly decreasing.  Consider the top row of $\Gamma(T)$: this has the smallest entry in each column, which in turn controls the smallest entries in each column of $\gamma(T)$ via Step 2.  If $j$ is the largest entry in the top row of $\Gamma(T)$, then the leftmost columns with $j$ on the top will have entries $1, \hdots, j-1$ in $\gamma(T)$.  Then the columns with $j-1$ on top in $\Gamma(T)$ will have entries $1, \hdots, j-2$ in $\gamma(T)$ and so on.  This guarantees that the bottom row of $\gamma(T)$ is weakly increasing, as well as parts of the of the higher rows.  Further, the columns are filled strongly increasing by construction.  By traversing the remaining rows of $\Gamma(T)$, we see that $\gamma(T)$ is semistandard.

Finally, this map is clearly an involution with $\gamma(\gamma(T))=T$.
\end{proof}
Note that the definition of $\gamma$ did not rely on $d$; thus, $\gamma$ is an automorphism of all constant-weight SSYTs that acts stratum-wise on $n$ and $d$.

\begin{thm}\label{thm:SSYTcompNai}
Fix $d < n$ and denote 
\[
\lambda_n^d=\underbrace{(n, \dots, n)}_{d \text{ copies}} \quad \textrm{and} \quad w_d^n =\underbrace{(d, \dots, d)}_{n \text{ copies}}.
\]
Let $T$ be an SSYT of shape $\lambda_n^d$ with constant weights $w_d^n$.  Further, let $\iota$ be the bijection from straight-type SSYT to integer GT patterns in Theorem~\ref{gt2ssyt}, $N$ the bijection from GT patterns with $\lambda_n^d$ and weights $w_d^n$ to GT patterns with $\lambda_n^{n-d}$ and weights $w_{n-d}^n$ in Theorem~\ref{prop:Naim_polytope}, and $\gamma$ the bijection from SSYT of shape $\lambda_n^d$ and weights $w_d^n$ to SSYT of shape $\lambda_n^{n-d}$ and constant weights $w_{n-d}^n$ from Theorem~\ref{tighttableaux}.

Then $N\left(\iota(T)\right) = \iota(\gamma(T))$.
\end{thm}

\begin{proof}\label{proof1}
Consider a tableau $T$ of shape $\lambda_n^d$ with constant weights $w_d^n$ and its associated GT pattern $\iota(T)$. 

First note that in the map $\iota$, 
\[
\lambda_{i,1}, \hdots, \lambda_{i,i}
\]
gives both the $i$th row of the triangular GT pattern in $\iota(T)$ and describes via the skew SSYT
\[
(\lambda_{i,1}, \hdots, \lambda_{i,i})/(\lambda_{i-1,1}, \hdots, \lambda_{i-1,i-1})
\]
exactly where the $i$'s are in $T$.

Let $\left((\mu_{i,j})_{j=1}^i\right)_{i=1}^{n-d}$ be the GT pattern for $N\left(\iota(T)\right)$.  Then
\[
\mu_{i,j} = \begin{cases}
\lambda_{n-i, d+j-i}, & \text{ for } j \leq i \leq d+j-1, \enskip j \leq n-d  \\
0, & \text{ for } j > n-d \\
n, & \text{ for } i > d+j-1,\enskip j \leq n-d \\
\end{cases}
\]

We now wish to compare the SSYTs $\iota^{-1}\left(N\left(\iota(T)\right)\right)$ and $\gamma(T)$. Both tableaux have shape $\lambda_n^{n-d}$ and weights $w_{n-d}^n$. We first consider the placement of the $1$'s, which need to be the $n-d$ left-hand entries of the bottom row due to shape and weights, but we analyze this case specifically to get a flavor of the arguments.  The $1$'s in $\iota^{-1}\left(N\left(\iota(T)\right)\right)$ are necessarily in the first $\mu_{1,1}=\lambda_{n-1,d}$ boxes of the bottom row, where $\lambda_{n-1,d}$ counts the number of boxes $\leq n-1$ in the $d$th row, i.e., the top row of $T$.  Since the top row of $T$ must contain $d$ $n$'s due to the weight vector and has $n$ boxes due to the shape, $\lambda_{n-1,d}=n-d$.  On the other hand, a column of $\gamma(T)$ has $1$ if and only if the corresponding column of $\Gamma(T)$ does not have a $1$ if and only if the corresponding column of $T$ does not have a $n+1-1=n$.  And a corresponding column of $T$ does not have an $n$ precisely when the largest value is $\leq n-1$, i.e., the column is in the first $\lambda_{n-1,d}$ columns of $T$.

Similarly, the locations of the $2$'s in $\iota^{-1}\left(N\left(\iota(T)\right)\right)$ are
\[
(\mu_{2,1},\mu_{2,2})/(\mu_{1,1}),
\]
where for $d, n-d\geq 2$,
\[
\mu_{2,1} = \lambda_{n-2,d-1} \quad \textrm{and} \quad \mu_{2,2} = \lambda_{n-2,d}
\]
We note that if $d=1$, then $T$ is just the single row with $n$ entries filled with $1, 2, 3, 4, 5$ in order, and if $n-d=1$, $\gamma(T)$ is that single row.  In either case, the complements are uniquely defined.

Returning to $d, n-d \geq 2$, $\mu_{2,2}= \lambda_{n-2,d}$ gives us the number of $2$'s in the second row (from the bottom) of $\iota^{-1}\left(N\left(\iota(T)\right)\right)$, where $\lambda_{n-2,d}$ counts the number of boxes in the top row of $T$ filled with entries from $[n-2]$.  Thus, the largest entries in the first $\lambda_{n-2,d}$ columns of $T$ are $\leq n-2$, meaning those columns do not have $n-1=n+1-2$.  So, in $\gamma(T)$, those columns must have $2$'s. And since $\lambda_{n-2,d}\leq \lambda_{n-1,d}=\mu_{1,1}$, those columns also already have $1$'s in $\gamma(T)$, hence those are precisely the $2$'s in the second row of $\gamma(T)$.

Now, $\mu_{2,1}-\mu_{1,1}$ is the number of $2$'s in the first row of $\iota^{-1}\left(N\left(\iota(T)\right)\right)$.  In $T$, $\lambda_{n-2,d-1}$ counts the number of entries in the second row from the top with values in $[n-2]$.  Since SSYT are strictly increasing, the remaining $n-\lambda_{n-2,d-1}$ entries must be $n-1$ (as no $n$'s can be outside of the top row).  However, if there is an $n-1$ in the $(d-1)$st row, there must be an $n$ in the $d$th row, meaning those columns are precisely the columns which have neither $1=n+1-n$ nor $2=n+1-(n-1)$ in $\gamma(T)$. So, the first $\lambda_{n-2,d-1}$ boxes of row $1$ of $\gamma(T)$ are filled with $1$'s and $2$'s, as desired.

By iteratively continuing this process, moving through the $\mu_{3,j}$'s to $\mu_{n-d,j}$'s, one obtains the following commutative diagram.

\begin{center}
\scalebox{1.2}{
\begin{tikzpicture}
  
  \node (A) at (0, 2) {$T$};
  \node (B) at (3, 2) {$\iota(T)$};
  \node (C) at (0, 0) {$\gamma(T)$};
  \node (D) at (3, 0) {$\iota(\gamma(T))=N(\iota(T))$};

  \draw[<->] (A) -- (B) node[midway, above] {$\iota$};
  \draw[<->] (B) -- (D) node[midway, right] {$N$};
  \draw[<->] (D) -- (C) node[midway, below] {$\iota$};
  \draw[<->] (C) -- (A) node[midway, left] {$\gamma$};
\end{tikzpicture}
}
\end{center}
\end{proof} 

\begin{remark}
We will call a particular tableaux or GT pattern under these isomorphisms, \emph{complements} of each other. 
\end{remark}

\begin{corollary}
Let $\Phi$ be a clearable equal norm tight frame and let $\Psi$ be a Naimark complement of $\Phi$. Then the SSYT $T$ associated to the eigensteps of $\Phi$ and the SSYT $R$ associated to the eigensteps of $\Psi$ satisfy $\gamma$
\[
\gamma(T) = R
\]
\end{corollary}

\begin{proof}
It was proven in \cite{HPpolytope} that $N$ maps between eigensteps of Naimark complements. Thus, the result immediately follows from Theorem~\ref{thm:SSYTcompNai}.
\end{proof}

Being both equal-norm tight and clearable means that the top row is equal to $\frac{\ell n}{d}$ and the row sums are increasing by $\ell$. Recall that this is turn places restrictions on what the associated tableau can look like. The shape of the tableaux will be a $d \times \frac{\ell n}{d}$ rectangle with weight vector $(\ell, \ell, \dots, \ell)$. Fortunately, these are the exact conditions required for Theorem~\ref{tighttableaux}. \\

\subsection{Generalized Naimark complements}\label{generalizednaimark}

We wish to be able to generalize the process in the preceding section to frames that are not necessarily either equal norm or tight. Results from \cite{king2025} prove the impossibility of a generalized Naimark complement being continuous, involutive, and Gale. However, a complement can satisfy any two of those conditions. In this work, we look at a generalized Naimark complement that is both continuous and involutive, but not Gale. This complement is a reformulation~\cite{king2025} of the generalized Naimark complement defined in~\cite{CFMevery}, where the reformulation makes the process more amenable to techniques on SSYT.  
\begin{definition}\label{defn:naimnontigh}
    Let $\Phi$ be $n$ vectors in $\bF^d$ with non-trivial span, where $B$ is the largest eigenvalue of $\Phi\Phi^\ast$.  Then we call $\Psi$ a generalized complement if $\Phi^\ast\Phi +\Psi^\ast\Psi = BI$. 
\end{definition}
This covers not only the case that $\Phi$ is a frame for $\bF^d$ but also when $\Phi$ is merely a frame for its span.  
In both cases, $B$ is the optimal upper frame bound. 

The involutive isomorphism $\gamma$ on rectangular SSYT in Theorem~\ref{tighttableaux} is a specific application of a map in \cite{GRgeneralized}, which in turn is a specific case of a more general involutive isomorphism on SSYT from \cite{RSpercent} called boxcomp.  Boxcomp will allow us to generalize Naimark complementation to non-rectangular SSYT.

\begin{definition}\label{boxcomp}
Let $T$ be a SSYT filled with $[n]$ and at most $c$ columns. $\bc(T)$ is the SSYT whose $j$th column (when viewed as a set) is the complement of $c+1-j$th column of $T$ in $[n]$. 
\end{definition}

Note that the Boxcomp map is proven to be an isomorphism in \cite{RSpercent}, but it is not difficult to derive that $\bc(T)$ will remain semistandard and is an involution.  Note also that boxcomp depends on the parameter $c$.  We will always set $c$ to be the number of columns of the SSYT.

\begin{example}
Consider the tableau $T$:
\[
\begin{ytableau}
5 \\
4 & 5 & 5 & 5 \\
2 & 2 & 3 & 3 & 5\\
1 & 1 & 1 & 2 & 2 
\end{ytableau}
\]
We have that $\bc(T)$ is 
\[
\begin{ytableau}
4 \\
3 & 4 & 4 & 4 \\
1 & 1 & 2 & 3 & 3\\
\end{ytableau}
\]
\end{example}

Notice that in boxcomp, since each column in $\bc(T)$ is the complement of a column in the tableau, the partition for each entry in $\bc(T)$ will precisely be the complement of that partition in T. Inspired by both boxcomp and the generalized Naimark, we define a map on polytopes of eigensteps.

\begin{definition}
    For nonnegative, non-increasing vectors $\lambda=(\lambda_1, \hdots, \lambda_n)$ and  $w =(w_1,\hdots, w_n)$ that satisfy Schur-Horn, let $\Lambda_{\lambda,w}$ denote the polytope of all triangular GT patterns (not necessarily integral) with the top row $\lambda$ and row sum differences $w$.
\end{definition}
\begin{thm}\label{generalGTcomp}
For nonnegative, non-increasing vectors $\lambda=(\lambda_1, \hdots, \lambda_n)$ and $w =(w_1,\hdots, w_n)$ that satisfy Schur-Horn and with $\lambda_n=0$, set $B=\lambda_1$.  Then there exists an involutive isomorphism
\[
\tilde{N}_{\lambda,w}: \Lambda_{\lambda,w} \longrightarrow \Lambda_{(B-\lambda_n, \hdots, B-\lambda_1),(B-w_1, \hdots, B-w_n)}
\]
given by
\[
\left(\tilde{N}_{\lambda,w}\left( ((\lambda_{k,\ell})_{\ell=1}^k)_{k=1}^n\right)\right)_{i,j} = 
B - \lambda_{i, i+1-j}.
\]
\end{thm} 

\begin{example}
Given the GT pattern below:
\[
\begin{matrix}
5 & & 5 & & 4 & & 1 & & 0\\
&5 & & 4 & & 1 & & 0\\
&&5 & & 4 & & 0\\
&&&5 & & 2\\
&&&&3
\end{matrix}\]
$\tilde{N}$ applied to this GT pattern yields 
\[
\begin{matrix}
5 & & 4 & & 1 & & 0 & & 0 \\
&5 & & 4 & & 1 & & 0 \\
&&5 & & 1 & & 0\\
&&&3 & & 0 \\
&&&&2
\end{matrix}
\]
Notice that this isomorphism is equivalent to subtracting each element from the largest entry and then reflecting horizontally.
\end{example}

\begin{proof}
Given $(\lambda_{k,\ell})_{\ell=1}^k)_{k=1}^n \in 
\Lambda_{\lambda, w}$, denote $\mu_{i,j} = \left(\tilde{N}_{\lambda,w}\left( ((\lambda_{k,\ell})_{\ell=1}^k)_{k=1}^n\right)\right)_{i,j}.$

First, we note that this is an involutive isomorphism from $\Lambda_{\lambda,w}$ to $\tilde{N}\left( \Lambda_{\lambda,w}\right)$.  Since $\lambda_1 = \lambda_{n,1}=B$ and $\lambda_n =\lambda_{n,n}=0$ by assumption, $\mu_{n,1}=B-0=B$ and $\mu_{n,n}=B-B=0$, so the $\mu$'s satisfy the given restrictions.  Further, 
\[
B-\mu_{i,i+1-j} = B-\left(B-\lambda_{i,i+1-(i+1-j)}\right) =\lambda_{i,j},
\]
meaning that $\tilde{N}$ is an involutive isomorphism onto its image.

Now want to show that the image is indeed $\Lambda_{(B-\lambda_n, \hdots, B-\lambda_1),(B-w_1, \hdots, B-w_n)}$. It immediately follows from the definition that the top row of the image is
\[
B-\lambda_n, \hdots, B-\lambda_1.
\]
So, we need to show that the interlacing inequalities are satisfied and the row sums are as desired. Since $\lambda_{i,j} \geq \lambda_{i,j+1}$, we have that $\mu_{i,j} =B- \lambda_{i,i+1-j} \geq B - \lambda_{i,i-j}=\mu_{i,j+1}$. Similarly, $\lambda_{i+1,j} \geq \lambda_{i,j}\geq \lambda_{i+1,j+1}$ implies $\mu_{i+1,j} \geq \mu_{i,j}\geq \mu_{i+1,j+1}$. 

Let $v_i$ denote the $i$th row sum of the $\mu_{i,j}$'s.  Note that $v_1 = \mu_{1,1} = B-\lambda_{1,1} = B-w_1$.  For $i>1$, we have
\begin{align*}
v_i &= \sum_{j=1}^i \mu_{i,j} - \sum_{k=1}^{i-1} \mu_{i-1,k}\\
&= \sum_{j=1}^i (B-\lambda_{i,i+1-j}) - \sum_{k=1}^{i-1} (B-\lambda_{i-1,i-k})\\
&= \left(iB -\sum_{j=1}^i \lambda_{i,j}\right) -\left((i-1)B + -\sum_{k=1}^{i-1} \lambda_{i-1,k} \right)\\
&= B - w_i,
\end{align*}
as desired.
\end{proof}

\begin{remark}
Like above we will call a particular tableaux or GT pattern under these isomorphisms, ``complements" of each other. 
\end{remark}

\begin{corollary} Fix $d<n$ and let $T$ be a SSYT of shape $\lambda = (\lambda_1, \dots \lambda_d)$ with weight $w = (w_1, \dots w_n)$. Further let $\iota$ be the bijection from straight-shape SSYT to interger GT patterns given in Theorem~\ref{gt2ssyt}, $\bc(T)$ be the isomorphism in Theorem~\ref{boxcomp} (with $c=\lambda_1$), and $\tilde{N}$ be the isomorphism given in \ref{generalGTcomp}. Then $\iota(\bc(T)) = \tilde{N}(\iota(T))$
\end{corollary}

\begin{proof}
Consider a tableaux $T$, with the underlying Young diagram $(\lambda_1, \lambda_2, \dots, \lambda_d)$. Recall from above that under $\iota$, 
\[
\lambda_{i,1}, \dots, \lambda_{i,i}
\]
gives both the $i$th row in the triangular GT pattern and describes where the $i$s are in the SSYT $T$ via the skew SSYT $(\lambda_{i,1}, \dots, \lambda_{i,i})/(\lambda_{i-1,1}, \dots, \lambda_{i-1,i-1})$. Additionally, similar to above let $(\gamma_{i,j})_{j=1})_{i=1}$ be the GT pattern for $\tilde{N}(\iota(T))$.

Like above, we will begin by comparing $\iota^{-1}(\tilde{N}(\iota(T)))$ and $\bc(T)$. $\iota(T)$ will have weight $w = (w_1, \dots w_n)$ and $\tilde{N}(\iota(T))$ will have 
weight $\tilde{w} = (n-w_1, \dots n-w_n)$. Since $\bc(T)$ is the set complement of the $c+1-j$th columns, $\bc(T)$ will have the same weight as well. 

We start with the placement of $1$'s. There will be a 1 in the $j$th column of $\bc(T)$ if and only if there is not a 1 in the $c+1-j$th column of $T$. Let the number of 1's in $T$ be $\lambda_{1,1}$. Since $T$ is semistandard, there will be a 1 in the first $n-\lambda_{1,1}$ columns of $\bc(T)$. This is exactly the placement of 1's in $\iota^{-1}(\tilde{N}(\iota(T)))$ as well. We can continue with this process, in a similar fashion to the proof for Theorem~\ref{gt2ssyt} for the remaining entries of $\bc(T)$ and $\iota^{-1}(\tilde{N}(\iota(T)))$. 

For the 2's, $\gamma_{2-1}- \gamma_{1,1}$ gives the number of 2's in the first row of $\iota^{-1}(\tilde{N}(\iota(T)))$. Since $\gamma_{2,1} = n - \lambda_{2,2}$ and $\gamma_{1,1} = n - \lambda_{1,1}$, we additionally know that the number of 2's in the first row is given by $\lambda_{1,1}- \lambda_{2,2}$. 

On the other hand there will be 2 in the first row of $\bc(T)$ if there is not a 1 already in the $j$th column of $\bc(T)$. Equivalently, there will be a 2 in the first row of $\bc(T)$ precisely in the columns where there is a 1 in the first row of $T$. Thus the number 2's in the first row of $\bc(T)$ is precisely $\lambda_{1,1}- \lambda_{2,2}$ as well. Continually iterating this process throughout the entries of $\bc(T)$ yields the following commutative diagram:

\begin{center}
\scalebox{1.2}{
\begin{tikzpicture}
  
  \node (A) at (0, 2) {$T$};
  \node (B) at (5, 2) {$\iota(T)$};
  \node (C) at (0, 0) {$\bc(T)$};
  \node (D) at (5, 0) {$\iota(\bc(T))=\tilde{N}(\iota(T))$};

  \draw[<->] (A) -- (B) node[midway, above] {$\iota$};
  \draw[<->] (B) -- (D) node[midway, right] {$\tilde{N}$};
  \draw[<->] (D) -- (C) node[midway, below] {$\iota$};
  \draw[<->] (C) -- (A) node[midway, above, sloped] {$\bc$};
\end{tikzpicture}
}
\end{center}
\end{proof}

\begin{corollary}
Let $\Phi=(\phi_i)_{i=1}^n$ be such that the largest eigenvalue of $\Phi\Phi^\ast$ is $B>0$. Additionally, let $\Psi =  (\psi_i)^n_{i=1}$ satisfy $\Phi^\ast\Phi+\Psi^\ast\Psi = BI$. If $\Phi$ has eigensteps $((\lambda_{i,j})_{j=1}^i)_{i=1}^n$, then $\tilde{N}(\lambda_{i,j})$ yields the eigensteps associated to $\Psi.$
\end{corollary}

\begin{proof}
Let $\Phi$ be such that the largest eigenvalue of $\Phi\Phi^\ast$ is $B>0$. Further let $\Psi$ 
satisfy $\Phi^\ast\Phi+\Psi^\ast\Psi = BI$. 

This implies that for any eigenvalue $\lambda_\Psi$ of $\Psi^\ast \Psi$ and $\lambda_\Phi $ of $\Phi^\ast \Phi$, we have $\lambda_\Psi = B -\lambda_\Psi$. It is clear to see that this applies to the eigensteps of $\Psi^\ast \Psi $ and $\Phi^\ast \Phi$ as well. Thus if $\lambda_{i,j}$ is an eigenstep of $\Phi$, then $B - \lambda_{i,j}$ is an associated eigenstep of $\Psi$. However, ordering the eigensteps of $\Psi$ from greatest to least yields that the ordering reverses. Thus, we have that the eigenstep of $\Psi$ are given by $B - \lambda_{i,i+1-j}$. This is exactly encoded by the $\tilde{N}$ and further $\bc(T)$ as well.
\end{proof}

Since applying the definition of the generalized Naimark complement Definition~\ref{defn:naimnontigh} to a tight frame trivially reduces to the standard Naimark complement Definition~\ref{defn:naimarktight}, we have the following.
\begin{cor}
The restriction of $\tilde{N}$ from Theorem~\ref{generalGTcomp} to $\Lambda_{n,d}$ is $N$ from Proposition~\ref{prop:Naim_polytope}.  The restriction of $\bc$ from Definition~\ref{boxcomp} to SSYT of shape $\underbrace{(n, \dots, n)}_{d \text{ copies}}$ and weights $w_d^n =\underbrace{(d, \dots, d)}_{n \text{ copies}}$ is $\gamma$ from Theorem~\ref{tighttableaux}.
\end{cor}
Note that it is not immediately obvious that these maps should be the same since one involves permuting entries and one involves changing the values and permuting the entries.

\section{Final Remarks}
In summary, we introduce a family of frames with rational eigensteps that we call clearable.  This allows us to represent the frames as semistandard Young tableaux and connect results from the frame theory community to results from the SSYT community.  We hope that this connection will be explored further.  In particular, the connection allows an approach to generating allowable eigensteps that is discretized and thus arguably simpler than the TopKill algorithm.  Also, Naimark complementation and generalized Naimark complementation may be interpreted as known involutive isomorphisms on SSYTs.

\subsection{Acknowledgments} We'd like to thank Maria Gillespie and Andrew Reimer-Berg for fruitful discussion on tableaux.  We'd also like to thank Frank Sottile for first alerting the frame theory community to the connection between eigensteps and GT patterns.
\printbibliography
\end{document}